\documentclass[12pt]{article}
\usepackage{amsmath,amsfonts,latexsym, xspace,amsthm,graphicx, amssymb}
\usepackage{enumerate}
\usepackage[letterpaper,left=1in,right=1in,top=1in,bottom=1in]{geometry}

\usepackage{subcaption}

\usepackage{url}
\usepackage{float}%%%%added this
\usepackage{mathtools}
\usepackage{color}

\newcommand{\ind}[1]{{\bf 1}_{#1}}
\newcommand{\E}[1]{\mathbb{E} \left(#1\right)}
\newcommand{\Ex}[2]{\mathbb{E}_{#1} \left(#2\right)}
\newcommand{\Var}[1]{\text{Var}(#1)}

\newcommand{\brac}[1]{\left(#1\right)}
\def\AA{\mathcal{A}}
\def\BB{\mathcal{B}}

\def\EE{\mathcal{E}}
\def\FF{\mathcal{F}}
\def\GG{\mathcal{G}}
\def\HH{\mathcal{H}}

\def\JJ{\mathcal{J}}
\def\KK{\mathcal{K}}
\def\LL{\mathcal{L}}

\def\PP{\mathcal{P}}

\def\ol{\overline}

\def\F{\Phi}

\def\d{\delta}
\def\D{\Delta}
\def\e{\varepsilon}
\def\f{\phi}

\def\G{\Gamma}

\def\th{\theta}

\def\l{\lambda}
\def\m{\mu}

\def\r{\rho}

\def\t{\tau}
\def\om{\omega}
\def\OM{\Omega}
\newcommand\Prob[1]{{\mbox{Pr}\left\{#1\right\}}}

\newcommand\beq[2]{\begin{equation}\label{#1}#2\end{equation}}

\newtheorem{lemma}{Lemma}
\newtheorem{theorem}{Theorem}
\newtheorem{corollary}{Corollary}
\newtheorem{definition}{Definition}
\newtheorem{claim}{Claim}

\newcommand{\bfrac}[2]{\left({\frac{#1}{#2}}\right)}

\newcommand{\set}[1]{\left\{#1\right\}}
\def\Tc{T_{\mbox{cov}}}

\newcounter{rot}%\addtocounter{rot}{1}, \therot

\title{The cover time of a biased random walk on a random cubic graph}

\author{Colin Cooper\thanks{Department of  Informatics,
King's College, University of London, London WC2R 2LS, UK.
Research supported in part by EPSRC grants EP/J006300/1 and EP/M005038/1.}
\and Alan
Frieze\thanks{Department of Mathematical Sciences, Carnegie Mellon
University, Pittsburgh PA 15213, USA.
Research supported in part by NSF grant DMS0753472.
}\and Tony Johansson\thanks{Department of Mathematics, Uppsala University, Sweden. Partly supported by the Knut and Alice Wallenberg Foundation.}
}

\begin{document}
\maketitle
\begin{abstract}
We study a random walk that prefers to use unvisited edges in the context of random cubic graphs. We establish asymptotically correct estimates for the vertex and edge covertimes, these being $\approx n\log n$ and $\approx \frac32n\log n$ respectively.
\end{abstract}
\section{Introduction}
Our aim in this paper is to analyse a variation on a simple random walk that may tend to speed up the cover time of a connected graph. This variation is just one of several possible approaches which include  (i) non-bactracking walks, see Alon, Benjamini, Lubetzky and Sodin \cite{Alon},  (ii) walks that prefer unused edges, see Berenbrink, Cooper and Friedetzky \cite{bcf} or (iii) walks that a biassed toward low degree vertices, see Cooper, Frieze and Petti \cite{CFP} or any number of other ideas. In this paper we study idea (ii) in the context of random cubic graphs, partially solving a problem left from \cite{bcf}.

\subsection{Unvisited Edge Process}
The papers \cite{bcf}, \cite{OS} describe a modified random  walk $X=(X(t),\; t \ge 0)$ on a graph $G$, which uses unvisited edges when available at the currently occupied vertex. If there are {\em unvisited  edges} incident with the current vertex, the walk  picks one u.a.r. and make a transition along this edge. If there are no unvisited edges incident with the current vertex, the walk moves to a random neighbour.

In \cite{bcf} this walk was called an {\em unvisited edge process} (or  edge-process), and in \cite{OS},
a {\em greedy random walk}. For random  $d$-regular graphs where $d=2k$ ($d$ even), it was shown in \cite{bcf} that the edge-process has vertex cover time $\Theta(n)$, which is best possible up to a constant. The paper also gives an upper bound of  $O(n \om)$ for the edge cover time. The $\om$ term comes from the w.h.p. presence of small cycles (of length at most $\om$). In \cite{CFNet}, the constant for the vertex cover time was shown to be $d/2$.

\begin{theorem}\label{newedge}
Let $X$ be an unvisited edge-process on a random $d$-regular graph, $d$ even. For $d \ge 4$, the following  holds w.h.p. The vertex cover time of the edge-process is
$\Tc^V(G)\sim dn/2$.
\end{theorem}

The paper \cite{bcf} included the experimental data shown in Figure \ref{fig:covertime} for the performance of red-blue walks on odd degree regular graphs. Namely, for $d=3$ the cover time is $\Theta(n \log n)$ and decreases rapidly with increasing $d$. For $d$ even, the experiments confirm the cover time result of  Theorem \ref{newedge} that $\Tc^V(G)\sim dn/2$.

\begin{figure}[H]
  \centering
 \includegraphics{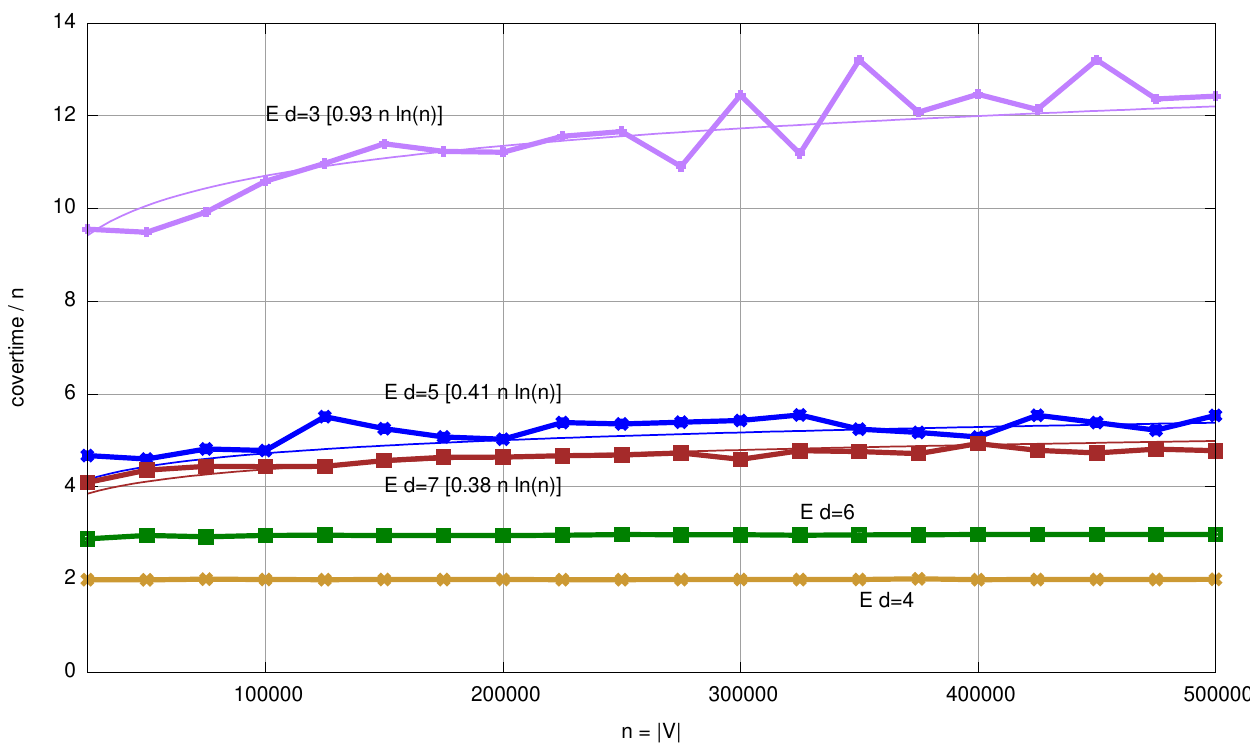}
   \caption{Normalised cover time of the unvisited edge process on random $d$-regular graphs as function of  $n=|V|$. Figure from \cite{bcf}}
  \label{fig:covertime}
\end{figure}

\subsection{Our results}
Let $G = (V, E)$ be a connected $3$-regular (multi)graph on an even number $n$ of vertices. Consider the following random walk process, called a {\em biased random walk}. It is an edge colored version of the previously described unvisited edge process. Initially color all edges red, and pick a starting vertex $v_0$. At any time, if the walk occupies a vertex incident to at least one red edge, then the walk traverses one of those red edges chosen uniformly at random, and re-colors it blue. If no such edge is available, the walk traverses a blue edge chosen uniformly at random. For $s\in\{1,\dots,n\}$ let $C_V(s)$ denote the number of steps taken by the walk until it has visited $s$ vertices, and similarly let $C_E(t)$ denote the number of steps taken to visit $t\in\{1,\dots,3n/2\}$ edges.

We will let $G$ be a random graph, and we use $\Ex{G}{X}$ to denote the expectation of $X$ with the underlying graph $G$ fixed.
\begin{theorem}\label{thm:whpcover}
Let $s,t$ be fixed such that $n-n\log^{-1}n \leq s \leq n$ and $(1-\log^{-2}n)\frac{3n}{2} \leq t\leq 3n/2$. Let $\e > 0$ also be fixed. Suppose $G$ is chosen uniformly at random from the set of $3$-regular graphs on $n$ vertices. Then with high probability, $G$ is connected and
\begin{align}
\Ex{G}{C_V(s)} & = (1 \pm \e)n\log \bfrac{n}{n-s+1} + o(n\log n), \label{vertexcover}\\
\Ex{G}{C_E(t)} & = \left(\frac32 \pm \e\right)n\log \bfrac{3n}{3n-2t+1} + o(n\log n).\label{edgecover}
\end{align}
\end{theorem}
Here $a = b\pm c$ is taken to mean $a \in [b-c, b + c]$. Note in particular that this shows that the expected vertex and edge cover times are asymptotically $n\log n$ and $\frac32 n\log n$ with high probability, respectively. The same statement is true with the word ``graphs'' replaced by ``configuration multigraphs''. Thus we have the following corollary.
\begin{corollary}\label{corr1}
W.h.p. the vertex cover time $\Tc^V(G)$ of $G$ is asymptotically equal to $n\log n$ and the edge cover time $\Tc^E(G)$ is asymptotically equal to $\frac32n\log n$.
\end{corollary}

It is of interest to compare the result of Corollary \ref{corr1} with other versions of random walk. Cooper and Frieze \cite{CFreg} showed that w.h.p. the vertex cover time of a random $d$-regular graph on $n$ vertices is asymptotically equal to $\frac{r-1}{r-2}n\log n$. The argument there also shows that the edge  cover time of a random $d$-regular graph on $n$ vertices is asymptotically equal to $\frac{r(r-1)}{2(r-2)}n\log n$. For $r=3$ these values are $2n\log n$ and $3n\log n$ respectively and are to be compared with $n\log n$ and $\frac32n\log n$. For a non-bactracking random walk, Cooper and Frieze \cite{CFNet} show that the vertex and edge cover times are asymptotically $n\log n$ and $\frac{r}{2}n\log n$ respectively. Interestingly, these values coincide with the results in Corollary \ref{corr1}.

\section{Outline proof of  Theorem \ref{thm:whpcover}}
We will choose the multigraph $G$ according to the configuration model. Each vertex $v$ of $G$ is associated with a set $\PP(v)$ of $3$ configuration points. We set $\PP = \cup_v \PP(v)$ and generate $G$ by choosing a pairing $\m$ of $\PP$ uniformly at random. The pairing $\m$ is exposed along with the biased random walk.

Starting at a uniformly random configuration point $x_1\in \PP$, we define $W_0 = (x_1)$. Given a walk $W_k = (x_1,x_2,\dots,x_{2k+1})$, the walk proceeds as follows. Set $x_{2k+2} = \m(x_{2k+1})$, thus exposing the value of $x_{2k+1}$ if not previously exposed. If $x_{2k+2}$ belongs to a vertex $v$ which is incident to some red edge (other than $(x_{2k+1},x_{2k+2})$ which is now recoloured blue), the walk chooses one of the red edges uniformly at random, setting $x_{2k+3}$ to be the corresponding configuration point. Otherwise, $x_{2k+3}$ is chosen uniformly at random from $\PP(v)$. Set $W_{k+1} = (x_1,\dots,x_{2k+3})$. We will refer to $x_1$ and $x_{2k+1}$ (and the vertices to which they belong) as the {\em tail} and {\em head} of $W_k$, respectively. We will also refer to $\set{x_1,x_2,\dots,x_{2k+1}}$ as the points of $\PP$ that have been {\em visited}.

Define partial edge and vertex cover times
\begin{align}
C_E(t) & = \min\{k : W_k \text{ spans $t$ edges}\}, \\
C_V(t) & = \min\{k : W_k \text{ spans $t$ vertices}\}.
\end{align}
We will mainly be concerned with the partial edge cover time, and write $C(t) = C_E(t)$ from this point on.

For $t \in \{1,2,\dots,\frac{3n}{2}\}$ we define a subsequence of walks by
\beq{tk}{
W(t) = W_{C(t)-1} = (x_1,x_2,\dots,x_{2k+1})
}
where $k$ is the smallest integer such that $|\{x_1,x_2,\dots,x_{2k+1}\}| = 2t - 1$. In other words, $W(t)$ denotes the walk up to the point when $2t-1$ of the members of $\PP$ have been visited. Thus throughout the paper:
\begin{itemize}
\item Time $t$ is measured by the number of edges $t$ that have been visited at least once.
\item The parameter $\d=\d(t)$ is given by the equation
\beq{defd}{
t=(1-\d)\frac{3n}{2}.
}
$\d(t)$ is  important as a measure of how close we are to the edge cover time.
\item the walk length $k$ is measured by the number of steps taken so far. Equation \eqref{tk} relates $t$ and $k$.
\end{itemize}

A 3-regular graph $G$ chosen u.a.r. is connected w.h.p. and we will implicitly condition on this in what follows. The bulk of the paper will be spent proving the following lemma.
\begin{lemma}\label{lem:Cstar}
For any fixed $\e > 0$ and $(1 - \log^{-2}n)\frac{3n}{2} \leq t \leq \frac{3n}{2}$,
\begin{equation}\label{eq:edgestar}
\E{C(t)} = \left(\frac{3}{2} \pm \e\right)n\log \bfrac{3n}{3n-2t+1} + o(n\log n)
\end{equation}
for $n$ large enough. Furthermore, for $n - \frac{n}{\log n} \leq s \leq n$,
\begin{equation}\label{eq:vertexstar}
\left(1-\e\right)n\log \bfrac{n}{n - s + 1} \leq
\E{C_V(s)}
\leq \left(1 + \e\right)n\log \bfrac{n}{n - s+1}.
\end{equation}
\end{lemma}
Expectations in Lemma \ref{lem:Cstar} are taken over the full probability space. In particular, if $\GG$ denotes the set of graphs,
$$
\frac32 n\log\bfrac{3n}{3n-2t+1} \approx \E{C(t)} = \frac{1}{|\GG|}\sum_{G\in\GG} \Ex{G}{C(t)}.
$$
In Section \ref{sec:whp} we strengthen Lemma \ref{lem:Cstar} to stating that almost every $G$ satisfies $\Ex{G}{C(t)} \approx \frac32 n\log(3n/(3n-2t+1))$ (and similarly for $C_V(s)$). Theorem \ref{thm:whpcover} follows.

An essential part of the proof of Lemma \ref{lem:Cstar} is a set of recurrences for the random variables $X_i(t)$, where $X_i(t)$ is the number of vertices incident with $i=0,1,2,3$ untraversed edges at {\em time} $t$, $t=1,2,...,3n/2$. We will argue that for most of the process, it takes approximately $3n/(3n-2t)$ steps of the walk to increase time by one. As the process finishes at time $3n/2$ we see that the edge cover time should be approximately
$$\sum_{t=1}^{3n/2}\frac{3n}{3n-2t+1}\approx \frac32n\log n,$$
as claimed in Corollary \ref{corr1}.

Now the recurrence for $X_3(t)$ has a solution that implies that $X_3(t)\approx n\d^{3/2}$, where $\d$ is as in \eqref{defd}. Given this, we would expect $X_3(t)$ to be zero when $\d$ is smaller than $n^{-2/3}$ or equivalently, when $3n/2-t$ is less than $n^{1/3}$. Thus we would expect that vertex cover time to be
$$\sum_{t=1}^{3n/2-n^{1/3}}\frac{3n}{3n-2t+1}\approx n\log n,$$
as claimed in Corollary \ref{corr1}.

We separate the proof of Lemma \ref{lem:Cstar} into phases. Define
$$\d_0=\frac{1}{\log\log n},\ \d_1 = \log^{-1/2}n,\ \d_2=\log^{-2}n,\ \d_3 = n^{-2/3}\log^4 n \text{ and }\d_4 = n^{-1}\log^{11}n$$
and set
\beq{tdef}{
t_i = (1-\d_i)\frac{3n}{2}\text{  for }i=0,1,2,3,4.
}
We do not hesitate to remind the reader of the meaning of these quantities.

The first phase, in which the first $t_1$ edges are discovered, will not contribute significantly to the cover time.
\begin{lemma}\label{lem:phaseone}
Let $\d_1 = \log^{-1/2} n$ and $t_1 = (1-\d_1)\frac{3n}{2}$. Then
$$
\E{C(t_1)} = o(n\log n).
$$
\end{lemma}
Between times $t_1$ and $t_4$ we bound the time taken between discovering new edges. The proof, in Section \ref{sec:covertime}, will be split into the ranges  $t_1\leq t\leq t_3$ and $t_3\leq t\leq t_4$.
\begin{lemma}\label{lem:increment}
Let $\e > 0$. For $t_1 \leq t \leq t_4$ and $n$ large enough,
$$
\E{C(t+1) - C(t)} = \left(3 \pm \e\right)\frac{n}{3n-2t} + O(\log n).
$$
\end{lemma}
Note that because $\frac{3n}2-t_1=O(\d_1n)$, the $O(\log n)$ term only contributes an amount $O(n\d_1\log n)=o(n\log n)$ to the the edge cover time.

Finally, the following lemma shows that the final $\log^{11}n$ edges can be found in time $o(n\log n)$.
\begin{lemma}\label{lem:lastedges}
For $t > t_4$ and $n$ large enough,
$$
\E{C(t) - C(t_4)} = o(n\log n).
$$
\end{lemma}
We note now that Lemma \ref{lem:Cstar} follows from Lemmas \ref{lem:phaseone}, \ref{lem:increment} and \ref{lem:lastedges}.
\section{Structural properties of random cubic graphs}
Here we collect some properties of random cubic graphs.
\begin{lemma}\label{lem:Grproperties}
Let $G$ denote the random cubic graph on vertex set $[n]$, chosen according to the configuration model. Let $\om$ tend to infinity arbitrarily slowly with $n$. Its value will always be small enough so that where necessary, it is dominated by other quantities that also go to infinity with $n$. Then with high probability,
\begin{enumerate}[(i)]
\item The second largest in absolute value of the eigenvalues of the transition matrix for a simple random walk on $G$ is at most $2\sqrt{2}/3 + \e\leq .99$ for some $\e$.
\item $G$ contains at most $\om3^\om$ cycles of length at most $\om$,
\item The probability that $G$ is simple is $\OM(1)$.
\end{enumerate}
\end{lemma}
For the proof of (i) see Friedman \cite{JF} and for the proof of (iii) see Frieze and Karo\'nski \cite{FK}, Theorem 10.3. Property (ii) follows from the Markov inequality, given that the expected number of cycles of length $k\leq \om$ can be bounded by $O(3^k)$.

Let $G(t)$ denote the random graph formed by the edges visited by $W(t)$. Let $X_i(t)$ denote the set of vertices incident to $i$ red edges in $G(t)$ for $i=0,1,2,3$. Let $\ol X(t) = X_1(t)\cup X_2(t)\cup X_3(t)$. Let $G^*(t)$ denote the graph obtained from $G(t)$ by contracting the set $\ol X(t)$ into a single vertex, retaining all edges. Define $\l^*(t)$ to be the second largest eigenvalue of the transition matrix for a simple random walk on $G^*(t)$.

We note that \cite[Corollary 3.27]{af}, if $\G$ is a graph obtained from $G$ by contracting a set of vertices, retaining all edges, then $\l(\G) \leq \l(G)$. This implies that $\l^*(t) = \l(G^*(t)) \leq \l(G) \leq 0.99$ for all $t$. Initially, for small $t$, we find that w.h.p. $G^*(t)$ consists of a single vertex. In this case there is no second eigenvalue and we take $\l^*(t)=0$. This is in line with the fact that a random walk on a one vertex graph is always in the steady state.

\section{Random walks and hitting times}\label{sec:hitting}
We are interested in calculating $C(t+1) - C(t)$, i.e. the time taken between discovering the $t$th and the $(t+1)$th edge. Between the two discoveries, the biased random walk can be coupled to a simple random walk on the graph induced by $W(t)$ which ends as soon as it hits a vertex of $\ol X$. In this section we derive the hitting time of a certain type of expanding vertex set.

Consider a simple random walk on a cubic graph $G = (V, E)$ with eigenvalue gap $1 - \l > 0$. For a set $S$ of vertices and a probability measure $\r$ on $V$, let $\Ex{\r}{H(S)}$ denote the expected hitting time of the set $S$, when the initial vertex is chosen according to $\r$. Let $\pi$ denote the stationary distribution of the random walk, uniform in the case of a regular graph and proportional to degrees in general. Let $P_u^{(t)}(v)$ denote the probability that a simple random walk starting at $u$ occupies vertex $v$ at step $t$ of the walk.

\begin{lemma}\label{lem:hitting}
Suppose $v$ is a vertex of a graph. Then the hitting time of $v$, starting from the stationary distribution $\pi$, is given by
$$
\Ex{\pi}{H(v)} = \frac{Z_{vv}}{\pi_v}
$$
where
$$
Z_{vv} = \sum_{t\geq 0} (P_v^{(t)}(v) - \pi_v).
$$
\end{lemma}
Lemma \ref{lem:hitting} can be found in \cite{af} (Lemma 2.11), and can be applied to hitting times of sets by contracting a set of vertices to a single vertex. The following bound will be frequently used. Suppose $G$ is a graph with eigenvalue gap $1-\l(G)$, and $S$ is a set of vertices in $G$. Then if $G_S$ is the graph obtained by contracting $S$ into a single vertex, retaining all edges, we have equal hitting times for $S$ in $G$ and $G_S$ and
\begin{equation}\label{eq:js}
\Ex{\pi}{H(S)} = \frac{n}{|S|} \sum_{t\geq 0} \left(P_{S}^{(t)}(S) - \pi_S\right) \leq \frac{n}{|S|}\sum_{t\geq 0} \l(G_S)^t = \frac{1}{1-\l(G_S)} \frac{n}{|S|} \leq\frac{1}{1-\l(G)}\frac{n}{|S|}.
\end{equation}
Indeed, $|P_v^{(t)}(v) -\pi_v| \leq \l^t$ for any $v,t$ in a graph with eigenvalue gap $1-\l$, see for example Jerrum and Sinclair \cite{JS} and use $j=k$ in the middle of the proof there of Proposition 3.1. Also, $\l(\G) \leq \l(G)$ for any $\G$ obtained from $G$ by contracting a set of vertices (see \cite[Corollary 3.27]{af}).

In the following lemma we implicitly view $G$ as a member of a sequence of graphs $(G_n)$, and $G$ having positive eigenvalue gap means that the second largest eigenvalue $\l_n$ of $G_n$ satisfies $\lim\sup \l_n < 1$.

Define $N_d(S)$ to be the set of vertices at distance exactly $d$ from a vertex set $S$. The set $S$ we consider induces $|S|/2$ edges all of which are 'far apart'.

\begin{lemma}\label{lem:treehitting}
 Let $G$ be a cubic graph on $n$ vertices with positive eigenvalue gap.
 Suppose $S$ is a set of vertices with $|S|\ge 2$ even such that $|S|=o(n)$ and
$$
|N_d(S)| = 2^d|S|
$$
for all $1\leq d\leq \om$, where $\om$ tend to infinity arbitrarily slowly with $n$. Then
$$
\Ex{\pi}{H(S)} \approx \frac{3n}{|S|}.
$$
\end{lemma}

\begin{proof}
We first note that the set $S$ contains exactly $|S|/2$ edges. Indeed, as $|N(S)| = 2|S|$ and the total degree of $S$ is $3|S|$, $S$ contains at most $|S|/2$ edges. As $|N_2(S)| = 4|S|$, each vertex of $N(S)$ must have exactly one edge to $S$, implying that $S$ contains at least $|S|/2$ edges.

Consider the graph $G_S$ obtained by contracting $S$ into a single node $s$, retaining all edges. In the graph $G_S$, $s$ has degree $3|S|$.  Then $s$ is a node with exactly $|S|/2$ self-loops, and is otherwise contained in no cycle of length at most $\om$, as $|N_d(S)| = 2^d|S|$ ensures that $G_S$ is locally a tree up to distance $\om$ from $s$. Since $\pi_s = |S| /n = o(1)$ we may choose $\om$ tending to infinity with $\om\pi_s = o(1)$. We have
$$
Z_{ss} = \sum_{t\geq 0} (P_s^{(t)}(s) - \pi_s) = \brac{\sum_{t =0}^\om P_s^{(t)}(s)} - o(1) + \sum_{t > \om} (P_s^{(t)}(s) - \pi_s).
$$
Repeating the argument following \eqref{eq:js},
$$
\sum_{t > \om} |P_s^{(t)}(s) - \pi_s| \leq \sum_{t>\om} \l^t = O(\l^\om) = o(1).
$$
We now argue that
$$
\sum_{t = 0}^\om P_s^{(t)}(s) = 3 + o(1).
$$
It is argued in Cooper and Frieze \cite{CFreg}, Lemma 7, that with no loops at vertex $s$, the expected number of returns to $s$ within $\om$ steps is $2+o(1)$. With the loops, when at $s$, there is a 1/3 chance of using a loop and so each visit to $s$ yields 3/2 expected returns; i.e. the 2 of \cite{CFreg} becomes $3=2\times 3/2$.
\end{proof}

We now expand Lemma \ref{lem:treehitting} to a larger class of sets.
\begin{definition}\label{def:rootset}
Let $G = (V, E)$ be a cubic graph. A set $S\subseteq V$ is a {\em root set of order $\ell$} if (i) $|S| \geq \ell^5$, (ii) the number of edges with both endpoints in $S$ is between $|S|/2$ and $(1/2 + \ell^{-3})|S|$, and (iii) there are at most $|S| /\ell^3$ paths of length at most $\ell$ between vertices of $S$ that contain no edges between a pair of vertices in $S$.
\end{definition}
Root sets of large order may be thought of as sets that ``almost satisfy the hypothesis of Lemma \ref{lem:treehitting}''. The following lemma shows  this definition is suitable for our purposes.
\begin{lemma}\label{lem:roothitting}
Let $\om$ tend to infinity arbitrarily slowly with $n$. Suppose $G$ is a cubic graph on $n$ vertices with positive eigenvalue gap, containing at most $\om^2$ cycles of length at most $\om$. If $S$ is a root set of order $\om$, then
$$
\Ex{\pi}{H(S)} \approx \frac{3n}{|S|}.
$$
\end{lemma}

\begin{proof}
Consider the contracted graph $G_S$, and let $s$ denote the contracted node. Then $s$ has degree $3|S|$, and $s$ has at most $(1/2 +2\om^{-3})|S|$ self-loops. Apart from the self-loops, $s$ lies on at most $|S|/\om^3$ cycles of length at most $\om$, as any cycle of $G_S$ containing $s$ corresponds to a path between members of $S$ in $G$.

Let $R = N(S)$, and note that $|R| = \OM(|S|)$. Consider the graph $\G$, defined as $G_S$ induced on the set of vertices at distance $1,2,\dots,\om$ from $s$. Note that $s$ is not included in $\G$. The graph $\G$ contains all of $R$, and as $s$ lies on at most $|S|/\om^3$ short cycles in $G_S$, the number of components in $\G$ containing more than one member of $R$ is $O(|S|/\om^3) = O(|R|/\om^3)$. As $G$ contains at most $\om^2$ short cycles, the number of components of $\G$ containing a cycle is at most $\om^2 = O(|R| / \om^3)$. This leaves $(1-o(1))|R|$ connected components in $\G$ which are all complete binary trees of height $\om$, each rooted at a member of $R$ and containing no other member of $R$. Let $T$ denote the set of vertices on such components.

Arbitrarily choose $|S|/2$ of the self-loops of $s$ in $G_S$, and designate them as {\em good}. Also say that an edge is good if it has both endpoints in $T\cup \{s\}$. All other edges are {\em bad}.

Consider a simple random walk $Z(\t)$ of length $\om$ on $G_S$, starting at $s$. Let $\BB_\t$ denote the event that $Z(\t)$ traverses a bad edge to reach $Z(\t+1)$. Whenever the walk visits $s$, the probability that it chooses a bad edge is $O(\om^{-3})$. If the walk is inside $T$, there are no bad edges to choose. So for any $\t \geq 0$ we have
$$
P_s^{(\t)}(s) = \Prob{Z(\t) = s\cap\ \bigcap_{r=0}^{\t-1} \overline{\BB_r}} + \Prob{Z(\t) = s \cap \bigcup_{r=0}^{\t-1} \BB_r} =
\Prob{Z(\t) = s\cap\ \bigcap_{r=0}^{\t-1} \overline{\BB_r}} + O(\om^{-2}).
$$
If $\BB_r$ does not occur for any $r \leq \t-1$, then the walk $(Z(0),\dots,Z(\t-1))$ can be viewed as the same Markov chain as considered in Lemma \ref{lem:treehitting}. So,
$$
\sum_{\t=0}^\om P_s^{(\t)}(s) = 3 + O(\om^{-1}).
$$
\end{proof}
We will argue that $X_1$ quickly makes up all but a $o(1)$ fraction of the vertices of $\ol X$.
For the purposes of discussion, we regard all edges of $G(t)$ visited exactly once by the walk $W(t)$ as coloured green.
In Section \ref{sec:uniform} we argue that the vertices of $X_1$ are, in a sense that will be made precise, uniformly distributed on those edges of $G(t)$ which are visited exactly once by $W(t)$.
%We can think of this as generating a walk $W'$, removing the vertices of $X_1$ (joining %the incident green edges), then re-inserting the vertices of $X_1$ uniformly at random %onto the green edges. (???)
In Section \ref{sec:concentration} we prove that toward the end of the walk, the number of green edges is significantly larger than $X_1$, which implies that most vertices of $X_1$ will be separated by a large distance. As $G$ contains few short cycles and $X_1$ makes up most of $\ol X$, this will imply that the number of vertices at distance $k$ from $\ol X$ is about $2^k|\ol X|$. In Section \ref{sec:hitting} we show how this implies that the expected hitting time of $\ol X$ in the simple random walk is approximately $3n/|\ol X|$, which by concentration (Lemma \ref{lem:concentration}) is about $3n/(3n-2t)$. Thus the partial edge cover time $C(t)$, the time to visit $t$ edges, will be given by
$$
\E{C(t)} = \sum_{r = 0}^{t-1} \E{C(r+1) - C(r)} \approx \sum_{r=0}^{t-1} \frac{3n}{3n-2r} \approx \frac32 n\log\bfrac{3n}{3n-2t}.
$$
In Section \ref{sec:uniform} we show how the vertices of $X_1$ are distributed within green edges. In Section \ref{sec:covertime} we use these results to calculate $\E{C(t+1) - C(t)}$ assuming certain concentration results, which are then proved in Section \ref{sec:concentration}.

\section{Random distribution of once-visited vertices}\label{sec:uniform}

Eventually the biased random walk will spend the majority of its time at vertices in $X_0$, i.e. vertices with no red incident edges. To bound the cover time, we will bound the time taken to hit $X_1\cup X_2$, which may be thought of as the boundary of $X_0$.

Let $W_k,k\geq 0$ denote the biased random walk after $2k+1$ walk steps have been taken. Say that a fixed finite walk $W$ is {\em feasible} if $\Prob{W_k = W} > 0$ for some $k \geq 0$, and fix a feasible walk $W$. Let $t$ be the time associated with $W$ as indicated in \eqref{tk}. Let $Y$ denote the subset of vertices in $X_1(t)$ that were visited and left exactly once by $W$. Note that $|Y\triangle X_1| \leq 1$, as the tail $v_0$ and head $v_k$ of the walk are the only vertices which may be in $X_1$ after being visited twice and then only when $v_0=v_k$. Indeed, the first time a vertex $v$ is visited, a feasible walk must enter and exit $v$ via distinct edges. Color all vertices of $Y$ green. We can write $Y = X_1(t)\setminus \{v_0\}$.

Given a feasible walk $W$, define a {\em green bridge} to be a part of the walk starting and ending in $V\setminus Y$, with any internal vertices being in $Y$. Note also that it is not necessary for a  green bridge to contain any vertices of $Y$. Form the {\em contracted walk} $\langle W\rangle$ by replacing any green bridge by a single green edge, with the walk orientation intact. Let $[W]$ denote the pair of (contracted walk, set), $[W]=(\langle W\rangle, Y)$, noting that $\langle W\rangle$ contains no vertex of $Y$.

We define an equivalence relation on the set of feasible walks by saying that $W\sim W'$ if and only if $[W] = [W']$. See Figure \ref{fig:equiv}. Thus the only way that $W,W'$ differ is as to where the vertices in $Y$ are placed on the green bridges. 

\begin{figure}[H]
  \begin{subfigure}{0.25\textwidth}
    \centering
    \includegraphics[width=\textwidth,page=1]{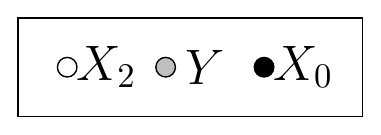}
  \end{subfigure}\\
  \begin{subfigure}{0.75\textwidth}
    \centering
    \includegraphics[width=\textwidth,page=2]{fig-walks}
  \end{subfigure}\\
   \begin{subfigure}{0.75\textwidth}
   \centering
    \includegraphics[width=\textwidth,page=3]{fig-walks}
  \end{subfigure}\\
   \begin{subfigure}{0.75\textwidth}
  \centering
    \includegraphics[width=\textwidth,page=4]{fig-walks}
  \end{subfigure}
  \caption{Two equivalent walks. Unvisited edges and vertices are not displayed, and edges visited exactly once are dashed.
  Lemma \ref{lem:uniform} shows that the walks are equiprobable.}
\label{fig:equiv}
\end{figure}

%The following lemma shows that equivalent walks are equiprobable.
\begin{lemma}\label{lem:uniform}
Let $k > 0$ and suppose $W$ is such that $\Prob{W_k = W} > 0$. If $[W] = (\langle W\rangle, Y)$ and $\langle W\rangle$ contains $\f$ green edges, then
$$
\Prob{W_k = W \mid [W_k] = [W]} = \frac{1}{|[W]|} = \frac{1}{(\f+|Y|-1)_{|Y|}},
$$
where $(a)_b=a(a-1)\cdots(a-b+1)$.
\end{lemma}
\begin{proof}
Let $W = (x_1,\dots,x_{2k+1})$ be a feasible walk on vertices $(v_1,\dots,v_{k+1})$. If $|\{x_1,\dots,x_{2k+1}\}| = 2t + 1$ and $|\{v_1,\dots,v_{k+1}\}| = r$ then for some $s$,
$$
\Prob{W_k = W} = \frac{1}{3n}\frac{1}{2^{r-1}}\frac{1}{3^{s}}\prod_{j=1}^{t+1} \frac{1}{3n-2j +1}.
$$
{\bf Explanation:} Let $W_k = (y_1,\dots,y_{2k+1})$ be the random walk. Firstly, $\Prob{y_1 = x_1} = 1/3n$. We reveal the matching $\m$ along with the walk. If $i$ is odd and $\m(y_i)$ is not previously revealed, then $\Prob{y_{i+1} = x_{i+1}} = 1/(3n-2j+1)$ where $|\{x_1,\dots,x_i\}| = 2j-1$. If $x_{i+1} \in \PP(v)$ for some $v$ that has not been visited by $W_k$ previously, then $\Prob{y_{i+2} = x_{i+2}} = 1/2$.
All other steps in the random walk either have probability $1$ or $1/3$. Here $s$ is the number of times the walk leaves a vertex that is in $X_0$ at the time of leaving. The value of $s$ is uniquely determined by $\langle W\rangle$.

Now suppose $v_i, 1 < i < k+1$ is a vertex which is visited by $W$ exactly once, so that $v_i\in Y$ and $x_{2i}, x_{2i+1}$ are visited only once. Let $v_j$ be another vertex, not necessarily in $Y$, and let $(x_{2j-1},x_{2j})$ be edge visited only once by $W$. Form the walk
$$
W' = (x_1,x_2,\dots,x_{2i-1},x_{2i+2},\dots,x_{2j-1},x_{2i}, x_{2i+1}, x_{2j},\dots,x_{2k-1}).
$$
Call this a {\em transposition of $W$}. Then, as the number of steps of the different types is unchanged,
\beq{climed}{
\Prob{W_k = W'} = \frac{1}{3n}\frac{1}{2^{r-1}}\frac{1}{3^s}\prod_{j=1}^{t+1}\frac{1}{3n-2j+1} = \Prob{W_k = W}.
}
Now, if $W\sim \widehat{W}$ for two feasible walks, then $\widehat{W}$ can be obtained from $W$ by a sequence of transpositions, and $\Prob{W_k = \widehat{W}} = \Prob{W_k = W}$. It is important to observe here that a transposition can move a vertex of $Y$ onto a green bridge that was previously empty of vertices in $Y$.

The walk $\langle W\rangle$ contains $\f$ green edges. We form any member of $[W]$ by distributing the vertices in $Y$ onto the green edges, assigning an internal order to each resulting green bridge. Let $Y = \{y_1,y_2,\dots,y_m\}$. We place $y_1$ on one of $\f$ green edges of $\langle W\rangle$, breaking the edge into two edges. There are then $\f + 1$ choices for the placement of $y_2$, and so on. This implies
$$
|[W]| = \f(\f+1)(\f+2)\dots(\f+|Y| - 1).
$$
\end{proof}

If $\langle W(t)\rangle$ contains $\f$ green edges $(e_1,e_2,\dots,e_\f)$, we let $(K_1,\dots,K_\f)$, $K_i \geq 1$, denote the lengths of the corresponding paths in $W(t)$. We remark that we can sample the vector $(K_1,\dots,K_\f)$ by a P\'olya urn process. Initially placing $\f$ balls of distinct colors in an urn, we repeat the following $|Y(t)|$ times: draw a ball uniformly at random, replace it in the urn and add another ball of the same color. The sizes of the resulting color classes, including the initial balls, are distributed as $(K_1,\dots,K_\f)$.
\section{Recurrences for $\E{X_i(t)}$}
We discuss the growth rate of $|X_i(t)|,i\geq 0$. In an abuse of notation, we will use $X_i(t),i\geq 0$ to denote both the set and its size. The context should dispel any possible ambiguity.

Let $\HH(t)$ denote the history of the process up until time $t$. We have $X_3(0)=n, X_i(0)=0, i=0,1,2$ and
\begin{align}
\E{X_3(t+1)\mid \HH(t)} & = X_3(t) -\frac{3 X_3(t)}{3n-2t+1}\\
\E{X_2(t+1)\mid \HH(t)}&= 1_{t=0}-\frac{2X_2(t)}{3n-2t+1}\\
\E{X_1(t+1)\mid \HH(t)}&= X_1(t)- \frac{2X_1(t)}{3n-2t+1}+ \frac{3 X_3(t)}{3n-2t+1}+\frac{2X_2(t)}{3n-2t+1}\label{X1t}\\
\E{X_0(t+1)\mid \HH(t)}&=n-(X_1(t)+X_2(t)+X_3(t)).
\end{align}
We remark that $X_2\leq 1$ and is almost irrelevant to the ensuing analysis. As justification for the above equations, consider $X_1$. There is a probability $\frac{3 X_3(t)}{3n-2t+1}$ that the newly paired point is associated with a vertex in $X_3(t)$. In which case, $X_1$ increases by one. Now a step involves visiting two members of $\PP$ and if the newly paired point is associated with a vertex $v\in X_1(t)$ then except in one exceptional case, the walk will move along an already visited edge $\set{v,w}$ where $w\in X_1(t)$ and then the second visited point will be in $\PP(w)$. The exceptional case is when $w\in X_2(t)$, which explains the last term in \eqref{X1t}. The other equations are explained similarly.

Thus,
$$\E{X_3(t)}=n\prod_{i=1}^t\brac{1-\frac{3}{3n-2i+1}}=n\exp\set{-\sum_{l=1}^\infty\frac{3^l}{l2^l} \sum_{i=1}^t \frac{1}{\brac{\frac{3n}{2}-i+\frac12}^l}}$$
Now
$$\int_{x=0}^{t-1}\frac{dx}{\brac{\frac{3n}{2}-x+\frac12}^l}\leq S_l=\sum_{i=1}^t \frac{1}{\brac{\frac{3n}{2}-i+\frac12}^l}\leq \int_{x=0}^{t}\frac{dx}{\brac{\frac{3n}{2}-x-\frac12}^l}
$$
which implies that if $\om_1=3n-2t\to\infty$ then
$$S_1=\log\bfrac{3n}{3n-2t}+O\bfrac{1}{3n-2t}\text{ and }\sum_{l=2}^\infty S_l=O(\om_1^{-1}).
$$
It follows that
\begin{align}
\E{X_3(t)} &= n \bfrac{3n-2t}{3n}^{3/2}\brac{1+O\bfrac{1}{3n-2t}}.\label{EX3}\\
\E{X_1(t)}&= 3n-2t-\E{3X_3(t)+2X_2(t)}\\
&= (3n-2t)\brac{1-\brac{1-\frac{2t}{3n}}^{1/2}}+O\bfrac{3n-2t}{n}^{1/2}. \label{EX1}
\end{align}
We also have some concentration around these values, as described in the following lemma which is proved in Section \ref{sec:concentration}.
\begin{lemma}\label{lem:concentration}
Let $0 < \e < 2/3$ and let $\om$ tend to infinity arbitrarily slowly. Let $\d = \d(t) = (3n-2t) / 3n$.
\begin{enumerate}[(i)]
\item\label{lem:concentration:X3} If $\om^{-1}\geq \d \geq \om n^{-2/3}$ then
$$
\Prob{|X_3(t) - n\d^{3/2}| \geq \frac{n\d^{3/2}}{\om^{1/2}}} = o(1),
$$
and if $\d \leq \om^{-1}n^{-2/3}$ then $|X_3(t)| = 0$ with high probability.
\item \label{lem:concentration:X1a}
Let $\d_1=\log^{-1/2}n$ and let $t_1=(1-\d_1)\frac{3n}{2}$.
$$
\Prob{\exists\,  3n/4 \leq t \leq t_1 : X_1(t) < \d n} = o(1).
$$
\item\label{lem:concentration:X1b} Let $\d_4 = n^{-1}\log^{11} n$.
$$
\Prob{\exists t_1\leq t \leq t_4 : |X_1(t) - 3n\d| \geq \om^{-1}\d n} = o(1).
$$
\item\label{lem:concentration:Phi} Set $\d_3 = n^{-2/3}\log^4 n$.
$$
\Prob{\exists t_1\leq t \leq t_3 : \F(t) < (\d_0\d)^{1/2}n} = o(1).
$$
where $\F(t)$ the number of green edges in $W(t)$.
\end{enumerate}
\end{lemma}
\section{Calculating the cover time}\label{sec:covertime}
\subsection{Early stages} \label{sec:phaseone}

With $t_1$ as in Lemma \ref{lem:concentration}, we show that $\E{C(t_1)} = o(n\log n)$. Suppose $W(t) = (x_1,x_2,\dots,x_{2k-1})$ for some $t$ and $k\geq 1$. If $x_{2k-1}\in \PP(\ol X(t))$ then $x_{2k} = \m(x_{2k-1})$ is uniformly random inside $\PP(\ol X(t))$, and since $C(t+1) = C(t) + 1$ in the event of $x_{2k}\in \PP(X_2\cup X_3)$, we have
\begin{equation}\label{eq:conditionincrement}
\E{C(t+1) - C(t)} \leq 1 + \E{C(t+1)-C(t) \mid x_{2k}\in \PP(X_1)}\Prob{x_{2k}\in \PP(X_1)},
\end{equation}
We use the following theorem of Ajtai, Koml\'os and Szemer\'edi \cite{AKS} to bound the expected change when $x_{2k}\in \PP(X_1)$.
\begin{theorem}\label{thm:aks}
Let $G = (V, E)$ be an $r$-regular graph on $n$ vertices, and suppose that each of the eigenvalues of the adjacency matrix with the exception of the first eigenvalue are at most $\l_G$ (in absolute value). Let $Z$ be a set of $cn$ vertices of $G$. Then for every $\ell$, the number of walks of length $\ell$ in $G$ which avoid $Z$ does not exceed $(1-c)n((1-c)r + c\l_G)^\ell$.
\end{theorem}
% Let $P(t) = \PP(X_1\cup X_2\cup X_3)$. If $x_{2k}\in \PP(X_1(t))$, the walk continues to $x_{2k+1},x_{2k+2},\dots,x_{2\ell}\notin P(t)$ until $x_{2\ell+1}\in P(t)\setminus\{x_{2k}\}$, yielding $C(t+1) - C(t) = \ell - k$. This is a simple random walk on $G$, starting at a uniformly random vertex of $v\in X_1(t)$, ending as soon as it hits a vertex of $X_1(t)\cup X_2(t)\cup X_3(t) \setminus \{v\}$. Note that while the walk must hit $X_1\cup X_2$ before it hits $X_3$, the bounds obtained from Theorem \ref{thm:aks} are improved by including $X_3$ in the target set.

The set $Z$ of Theorem \ref{thm:aks} is  fixed. In our case the exit vertex $u$ of the red walk is chosen randomly from $X_1(t)$. This follows from the way the red walk constructs the graph in the configuration model. The subsequent walk  now begins at  vertex $u$ and continues until it hits a vertex of $Y_u=X_1(t)\setminus\{u\}$ (or more precisely $Y_u \cup X_2(t)$). Because the exit vertex $u$ is random, the set $B_u=Y_u\cup X_2(t)\cup X_3(t)$ differs for each possible exit vertex $u \in X_1(t)$. To apply Theorem \ref{thm:aks}, we  split $X_1(t)$ into two disjoint sets $A, A'$ of (almost) equal size. For  $u \in A$, instead of considering the number of steps needed to hit $B_u$, we can upper bound this by the number of steps needed to hit $B'=A'\cup X_2\cup X_3$.

Let $Z(\ell)$ be a simple random walk of length $\ell$ starting from a uniformly chosen vertex of $A$. Thus $Z(\ell)$ could be any of $|A| 3^\ell$ uniformly chosen random walks. Let $c = |B'| / n$. The probability $p_\ell$ that a randomly chosen walk of length $\ell$ starting from $A$ has avoided $B'$ is at most
$$
p_\ell \leq \frac{1}{(|X_1(t)|/2)3^\ell}(1-c)n(3(1-c) + c\l_G)^\ell \leq \frac{2(1-c)n}{|X_1(t)|}((1-c) + c\l)^\ell,
$$
where $\l\leq.99$ (see Lemma \ref{lem:Grproperties}) is the absolute value of the second largest eigenvalue of the transition matrix of $Z$. Thus
\beq{AHC}{
\Ex{A}{H(B')} \leq \sum_{\ell\geq 1} p_\ell \leq \frac{2(1-c)n}{|X_1(t)|} \frac{1}{c(1-\l)}.
}
As $|B'| = |X_1|/2 + |X_3|$, we have
\begin{equation}\label{eq:aksincrement}
\E{C(t+1) - C(t) \mid x_{2k}\in \PP(X_1(t))} = O\bfrac{(n - |X_3|)n}{|X_1|\left(|X_1|+ |X_3|\right)}.
\end{equation}

{\bf Phase I: $t \leq 3n/4$.}

We use the bound $|X_3(t)| \geq n-t\geq n/4$ and the fact that
$$\Prob{x_{2k}\in \PP(X_1(t))} = \frac{|X_1(t)|}{3n-2t-1} = O\bfrac{|X_1(t)|}{n}.$$
Then \eqref{eq:conditionincrement} and \eqref{eq:aksincrement} imply
\begin{align}
\E{C(t+1) - C(t)} & \leq 1 +O\brac{\frac{tn}{|X_1(t)| n}\cdot \frac{|X_1(t)|}{n}} = 1 + O\bfrac{t}{n}.
\end{align}
Summing over $t \leq 3n/4$ gives $\E{C(3n/4)} = O(n)$.

{\bf Phase II: $3n/4 \leq t \leq (1-\d_1)3n/2$.}

It follows from Lemma \ref{lem:concentration} (\ref{lem:concentration:X1a}) that with high probability, for all $3n/4 \leq t \leq (1-\d_1)\frac{3n}{2}$, we have $|X_1(t)| \geq (3n-2t)/3$. In particular, $|X_1|(|X_1|+ |X_3|) = \OM((3n-2t)^2)$, and by \eqref{eq:aksincrement},
\begin{align}
\E{C\left((1-\d_1)\frac{3n}{2}\right) - C\bfrac{3n}{4}} & \leq \sum_{t = \frac{3n}{4}}^{(1-\d_1)\frac{3n}{2}} O\left[\frac{(n-|X_3|)n}{|X_1|\left(|X_1|+ |X_3|\right)}\right] \\
& =  O\left[\sum_{t=\frac{3n}{4}}^{(1-\d_1)\frac{3n}{2}}\frac{n^2}{(3n-2t)^2}\right] \\
& = O\bfrac{n}{\d_1}\\
&=o(n\log n).\label{onlogn}
\end{align}
\subsection{Later Stages}
We will now use Lemmas \ref{lem:roothitting} and \ref{lem:uniform}, together with Definition \ref{def:rootset} and Lemma \ref{lem:concentration}.

For $t$ with $\d \leq \log^{-1/2}n$ we set $\om = \om(t) = \log(-\log \d)$ and define the events (with $\ol X(t) = X_1(t) \cup X_2(t)\cup X_3(t)$)
\begin{align}
\AA(t) & = \{|X_1(t) - 3n\d| =O( \om^{-1}\d n)\}, \\
\BB(t) & = \{\ol X(t) \text{ is a root set of order } \om\}.
\end{align}
and set $\EE(t) = \AA(t) \cap \BB(t)$. As a consequence of Lemma \ref{lem:roothitting}, equation \eqref{AHC} and the fact that $\E{\ol X(t)} \approx 3n-2t$, we have
\begin{equation}\label{eq:condincrement}
\E{C(t+1) - C(t)} = (3\pm \e)\frac{n}{3n-2t}\Prob{\EE(t)} +  O\brac{\frac{n}{3n-2t}} \Prob{\ol{\EE(t)}} + O(\log n).
\end{equation}
Here the $O(\log n)$ and $\e$ terms account for the number of steps needed to take for the random walk Markov chain to mix to within variation distance $\e$ of the stationary distribution, at which time we apply Lemma \ref{lem:roothitting}. Here we rely on $\l^*(t) \leq 0.99$. In the event of $\ol{\EE(t)}$ we use the fact that $\ol X(t) =\Omega(3n-2t)$, which follows from Lemma \ref{lem:concentration}(ii) and the hitting time bound $\frac{1}{1-\l} \frac{n}{\ol X(t)}$ (see \eqref{eq:js}) to conclude that the hitting time is $O(n/(3n-2t))$.

Lemma \ref{lem:concentration} implies that $\AA(t)$ occurs with high probability for any fixed $t \geq 3n(1-\log^{-1/2}n)/2$ and we will argue in Sections \ref{sec:expansionviaconcentration} and \ref{sec:expansionwithoutconcentration} that $\BB(t)$ also occurs with high probability. Lemma \ref{lem:increment} will follow. Lemma \ref{lem:lastedges} is proved in Section \ref{seclastedges}  and Lemma \ref{lem:Cstar} follows.

\subsection{Expansion via concentration}\label{sec:expansionviaconcentration}

As discussed above, we are interested in showing that the event $\EE(t)$ occurs with high probability.
\begin{lemma}\label{lem:expansionviaconcentration}
Fix $t$ and let $\d=(3n-2t)/n$. If $\d_1=\log^{-1/2} n \geq \d \geq \d_3=n^{-2/3}\log^4 n$ then,
$$
\Prob{\EE(t)} = 1 - o(1).
$$
\end{lemma}

\begin{proof}
Fix some $t,\d$ in the given range. Expose $[W(t)]$. Lemma \ref{lem:concentration} shows that with high probability, $\F(t), |X_1(t)|$ satisfy
\begin{align}
\F(t) & \geq (\d_0\d)^{1/2}n, \\
|X_1(t)| & = 3\d n + O(\om^{-1}\d n).
\end{align}
As already remarked, this shows that $\Prob{\AA(t)} = 1-o(1)$. We next bound $|X_3(t)|$. To do this, we will only use the fact that it is dominated by $|X_1(t)|$ throughout this phase: as $\d \geq n^{-1/2}\log^4 n$ and $|X_1(t)|\approx 3\d n$, by Lemma \ref{lem:concentration}.
$$
|X_3(t)| = \d^{3/2}n + O(\om^{-1}\d^{3/2}n) = o(\d n) = o(|X_1(t)|)
$$
with high probability. We can now show that $\ol X(t) = X_1(t)\cup X_2(t) \cup X_3(t)$ is a root set of order $\om$ with high probability. Here $\om$ is chosen to satisfy \eqref{valom} below.

Let $E_t$ denote the set of $t$ edges discovered by the walk, and $E_t^c$  the set of (random) edges yet to be discovered. The number of edges inside $\ol X(t)$ is given by
\begin{align}
e(\ol X(t)) = |E_t^c| + |E(X_1\cup X_2)\cap E_t|
\end{align}
where $|E_t^c| = (X_1 + 2X_2+ 3X_3)/2$, so
$$
|E_t^c|=\frac{|X_1|}{2} + O(\d_1^{1/2}) =\frac{|X_1|}{2} + O(\om^{-3})
$$
for $\om^{3} \ll\d_0^{-1/2}$.

We bound the number of paths of length at most $\om$ between vertices of $X_1$ on edges of $E_t$, showing that the number is $O(|X_1|/\om^3)$. Note that such paths include $E(X_1)\cap E_t$, so that the bound implies $|E(X_1)\cap E_t| = O(|X_1|/\om^3)$.

Let $u,v\in X_1$. Suppose $u$ is placed on some green edge $f_1$. There are at most $3^\om$ green edges at distance at most $\om$ from $f_1$, so as $v$ is placed in a random green edge,
$$
\Prob{d(u, v) \leq \om} = O\bfrac{3^\om}{\F} = O\bfrac{3^\om}{n(\d_0\d)^{1/2}}.
$$
So the expected number of pairs $u,v\in X_1$ at distance at most $\om$ is bounded by
\begin{equation}\label{eq:distancebound}
\sum_{u, v\in X_1} \Prob{d(u, v)\leq \om} = O\bfrac{|X_1|^23^\om}{n(\d_0\d)^{1/2}} = O(n\d_0^{-1/2}\d^{3/2}3^\om) = o(|X_1|/\om^3),
\end{equation}
if we choose
\beq{valom}{
\om^33^\om\ll(\d_0/\d)^{1/2}.
}
 With high probability the number of paths is $O(|X_1|/\om^3)$ by the Markov inequality. This shows that $\ol X(t)$ is a root set of order $\om$ with high probability.
\end{proof}

\subsection{Maintaining expansion without concentration} \label{sec:expansionwithoutconcentration}
\begin{lemma}\label{lem:expansionwithoutconcentration}
For any fixed $t$ such that $\d=\d(t)$ satisfies $\d_4=n^{-1}\log^{11}n\leq \d\leq \d_3=n^{-2/3}\log^4 n$,
$$
\Prob{\EE(t)} = 1 - o(1).
$$
\end{lemma}
\begin{proof}
By Lemma \ref{lem:concentration}, at time $t_3 = (1-\d_3)\frac{3n}{2}$ the sizes of $X_1, X_3$ and $\F$ satisfy the following with high probabilty,
\begin{align}
\F(t_3) & \geq n(\d_0\d_3)^{1/2}, \\
|X_1(t_3)| & \approx 3n\d_3, \\
|X_3(t_3)| & \approx n\d_3^{3/2} = O(\log^6 n).
\end{align}
For $v\in X_1$ let $N_\ell^0(v)$ denote the number of vertices of $X_0$ at distance $\ell$ from $v$, using only edges of $E_t$. Let $X_1'\subseteq X_1$ denote the set of vertices $v\in X_1$ with $|N_\ell^0(v)| = 2^\ell$ for all $\ell\leq \om$, and let $X_1'' = X_1\setminus X_1'$. We have $|X_1''(t_3)| = O(n\d_3^{3/2}3^\om) = O(\log^7 n)$ with high probability from \eqref{eq:distancebound}. By Lemma \ref{lem:concentration} \ref{lem:concentration:X1b} we have $X_1(t) \geq (1 - o(1))\log^{11} n$ for $\d \geq \d_4$. So for $t_3 \leq t \leq t_4$ we have $|X_1''(t)| \leq |X_1''(t_3)| = O((\log n)^{11} / \om^3) = O(|X_1(t)| / \om^3)$ w.h.p. This shows that $\ol{X}(t)$ is a root set of order $\om$.
\end{proof}
\subsection{The final edges: Proof of Lemma \ref{lem:lastedges}}\label{seclastedges}

Recall that $\d_4 = n^{-1}\log^{11}n$ and $t_4 = (1-\d_4)\frac{3n}{2}$. So far we have shown that the edge cover time claimed by Lemma \ref{lem:Cstar} holds for all $t_1\leq t\leq t_4$. We now show that
$$
\E{C\bfrac{3n}{2} - C(t_4)} = o(n\log n).
$$
Fix $t_4 \leq t < \frac{3n}{2}$. We bound the hitting time of $\ol X(t) = X_1(t)\cup X_2(t) \cup X_3(t)$, which has size $3n-2t$. We contract $\ol X(t)$ into a single node $x$ and apply Lemma \ref{lem:hitting}, using the bound
$$Z_{xx} \leq \sum_{t\geq 0}|P_x^{(t)}(x) - \pi_x| \leq \sum_{t\geq 0}\l^t = \frac{1}{1-\l} = O(1).$$
As the random walk Markov chain mixes to within $\e$ total variation distance of $\pi$ in $O(\log n)$ steps, it follows that
$$
\E{C(t + 1) - C(t)} \leq O(\log n) + \frac{1}{1-\l}\frac{n}{3n-2t}.
$$
So, as $3n/2 - t_4 = O(\log^{11} n)$,
$$
\E{C\bfrac{3n}{2} - C(t_4)} \leq O(\log^{12} n) + \frac{1}{1-\l}\sum_{t = t_4}^{\frac{3n}{2} - 1}\frac{n}{3n-2t} = O\left(n\log\left(\log^{11} n\right)\right) = o(n\log n).
$$

\subsection{Strengthening to ``with high probability''}\label{sec:whp}

So far, all expectations are taken over the full probability space of random graphs and random walks, simultaneously generated. In particular, the expected cover time is the average cover time of all cubic multigraphs. In this section we prove that almost all cubic multigraphs have the same cover time.

Let $\GG$ denote the set of $3$-regular (multi)graphs. For $G\in \GG$ and a random variable $X$, write
$$
\Ex{G}{X} = \E{X\mid G},
$$
so that, as $G\in \GG$ is chosen uniformly at random,
$$
\E{X} = \frac{1}{|\GG|} \sum_{G\in \GG} \Ex{G}{X}.
$$
We prove the following lemma. Define $\d_2 = \log^{-2}n$ and $t_2 = (1-\d_2)\frac{3n}{2}$.
\begin{lemma}\label{lem:whp}
Let $t > t_2$. If $G\in \GG$ is chosen uniformly at random, then with high probability,
$$
\Ex{G}{C(t)} = \left(\frac32\pm\e\right)n\log \bfrac{3n}{3n-2t+1} + o(n\log n).
$$
\end{lemma}

\begin{proof}[Proof of Lemma \ref{lem:whp}]
Fix some $t > t_2$. We define the following subsets of $\GG$, with $\e > 0$ arbitrary and $\om$ to be defined shortly,
\begin{align}
\HH & = \left\{G\in \GG : \left|\Ex{G}{C(t) - C(t_2)} - \sum_{s=t_2}^{t-1}\frac{3n}{3n-2s}\right| \geq \e n\log \bfrac{3n}{3n-2t+1}\right\}, \\
\JJ & = \left\{G\in \GG : \Ex{G}{C(t_2)} \geq \frac{n\log n}{\om}\right\}, \\
\KK & = \left\{G\in \GG : \max_{t_1\leq t\leq t_4} \frac{|X_1(t) - (3n-2t)|}{3n-2t} \geq \e\right\}, \\
\LL & = \left\{G \in \GG : \l(G) > 0.99\right\}.
\end{align}
We will show that asymptotically, the union of these four sets has size $o(|\GG|)$. In particular, almost all $G\in\GG$ are in $\ol{\HH\cup \JJ\cup \LL}$, which implies that w.h.p. $\Ex{G}{C(t)}$ has the desired value.

It follows from Lemma \ref{lem:Grproperties} that $|\LL|/|\GG| = o(1)$. Also, Lemma \ref{lem:secondmoment} proved below shows that $|\KK|/|\GG| = o(1)$.

In Section \ref{sec:phaseone} (see \eqref{onlogn}) we show that $\E{C(t_1)} = o(n\log n)$. From Lemma \ref{lem:expansionviaconcentration} and \eqref{eq:condincrement},
$$
\E{C(t_2)} = \E{C(t_1)} + O\left(\sum_{t = t_1}^{t_2}\frac{3n}{3n-2t}\right) = o(n\log n) + O\left(n\log\bfrac{\d_1}{\d_2}\right) = o(n\log n).
$$
So $\E{C(t_2)} \leq \om^{-2}n\log n$ for some $\om$ tending to infinity. Then
$$
\frac{n\log n}{\om^2} \geq \E{C(t_2)} = \frac{1}{|\GG|}\sum_{G\in\GG} \Ex{G}{C(t_2)} \geq \frac{|\JJ|}{|\GG|} \frac{n\log n}{\om},
$$
which implies that $|\JJ|/|\GG| \leq \om^{-1} = o(1)$. Define $\GG' = \GG\setminus(\JJ\cup \KK\cup \LL)$, and $\HH' = \HH\cap \GG'$. We will show that $|\HH'|/|\GG'| = o(1)$, which will imply that $|\HH|/|\GG| = o(1)$.

Write $\D (s) = C(s+1) - C(s)$. Define for $G\in\GG'$,
$$
T_G = \left\{t_1 \leq s < t : \left|\Ex{G}{\D (s)} - \frac{3n}{3n-2s}\right| \geq \frac{\e n}{3n-2s}\right\}.
$$
Define, for some $\om$ tending to infinity with $n$,
$$
\FF' = \left\{G \in \GG' : \sum_{s\in T_G} \frac{n}{3n-2s} \geq \frac{1}{\om} n\log \bfrac{3n}{3n-2t+1}\right\}.
$$
If $G\in\GG'$ then $\Ex{G}{\D(s)} \leq (1-\l(G))^{-1} n/|X_1| \leq 100n/(3n-2s)$ for all $s$, and so if $G\in \GG'\setminus \FF'$,
\begin{align}
\left|\Ex{G}{C(t) - C(t_1)} - \sum_{s=t_1}^{t-1}\frac{3n}{3n-2s}\right|
& \leq \e\sum_{s\notin T_G}\frac{n}{3n-2s} + \sum_{s\in T_G} \left|\Ex{G}{\D(s)} - \frac{3n}{3n-2s}\right| \\
& \leq \frac12\e n\log\bfrac{n}{3n-2t+1} + 103\sum_{s\in T_G} \frac{n}{3n-2s} \\
& \leq\e n\log \bfrac{3n}{3n-2t+1}.
\end{align}

In particular, $\HH'\subseteq \FF'$, and it remains to argue that $|\FF'| / |\GG'| = o(1)$.

For $t_1 \leq s < t$ let $\FF_s'\subseteq \GG'$ denote the set of graphs $G$ with $s\in T_G$. Since $\Prob{\EE(s)} = 1 - o(1)$ (Lemmas \ref{lem:expansionviaconcentration} and \ref{lem:expansionwithoutconcentration}), almost all $G\in \GG'$ are such that $\Prob{\EE(s) \mid G} = 1-o(1)$. For such a $G$ we have (as in \eqref{eq:condincrement} with $\e$ replaced by $\e/2$)
\begin{align}
\Ex{G}{\D(s)} & = \brac{3\pm\frac{\e}2}\frac{n}{3n-2s}\Prob{\EE(s)\mid G} + O\left(\frac{1}{1-\l}\frac{n}{3n-2s}\right)\Prob{\ol{\EE(s)}\mid G} + O(\log n) \\
& = (3\pm \e)\frac{n}{3n-2s}, \label{eq:EDs}
\end{align}
and so $s\notin T_G$. Note that we have used $s\geq t_2$ here, in order to eliminate the $O(\log n)$ term. This shows that $|\FF_s'| / |\GG'| = o(1)$ for each $t_1 \leq s < t$. Let $\om$ tending to infinity be such that $|\FF_s'|/|\GG'| \leq \om^{-2}$ for all $s$. So,
\begin{align}
\frac{1}{|\GG|} \sum_{G\in\GG}\sum_{s\in T_G}\frac{n}{3n-2s} = \sum_s \frac{|\FF_s|}{|\GG|} \frac{n}{3n-2s} \leq \frac{3}{2\om^2} n\log\bfrac{3n}{3n-2t+1}.
\end{align}
But by definition of $\FF'$,
$$
\frac{1}{|\GG|}\sum_{G\in\GG}\sum_{s\in T_G} \frac{n}{3n-2s} \geq \frac{|\FF'|}{|\GG|} \frac{1}{\om} n\log \bfrac{3n}{3n-2t+1}
$$
and we conclude that $|\FF'| / |\GG| \leq \frac32\om^{-1} = o(1)$. This finishes the proof.
\end{proof}

\subsection{The vertex cover time}\label{sec:vertexcover}

Using Lemma \ref{lem:concentration} (\ref{lem:concentration:X3}) we can express the partial vertex cover time in terms of the partial edge cover time.
\begin{lemma}
Let $n - n\log^{-1} n \leq s \leq n$ be fixed. Then with high probability, $G$ is such that
\begin{equation}\label{vcover}
 \Ex{G}{C_V(s)} = (1\pm \e)n\log\bfrac{n}{n-s+1}.
\end{equation}
\end{lemma}
\begin{proof}
As $X_3(t)$ is the set of undiscovered vertices at time $C(t) - 1$, we can write $C_V(s) = C(\t_s)$, where
$$
\t_s = \min\left\{t : |X_3(t)| = n - s\right\}.
$$
Define $\d_s = (1 - (s-1)/n)^{2/3}$. Let $\t_s^- = (1-\d_s\om)\frac{3n}{2}$ for some $\om = O(\log\log n)$ tending to infinity with $n$. By Lemma \ref{lem:concentration}, with high probability
$$
|X_3(\t_s^-)| \approx n\left(\d_s \om\right)^{3/2} = n\left(1-\frac{s-1}{n}\right)\om^{3/2} \approx (n-s)\om^{3/2} \gg n-s,
$$
so $\t_s > \t_s^-$ with high probability. By a similar calculation, if $\t_s^+ = (1-\d_s\om^{-1})\frac{3n}{2}$ then $\t_s < \t_s^+$ with high probability. Lemma \ref{lem:concentration} implies that $t_2< \t_s\leq \frac{3n}{2}$ w.h.p. and so Section \ref{sec:whp} implies that w.h.p., $G$ is such that $\Ex{G}{C(\t_s^\pm)} \approx \frac32 n\log(3n/(3n-2\t_s^\pm + 1))$. So,
\begin{align}
\Ex{G}{C_V(s)} & = \Ex{G}{\Ex{G}{C(\t_s) \mid \t_s}} \\
 & \leq \Ex{G}{\left(\frac32 + \e\right)n\log\bfrac{3n}{3n-2\t_s+1}} \\
& \leq \Prob{t \leq \t_s^+} \left(\frac32 + \e\right)n\log\bfrac{3n}{3n-2\t_s^++1} + \Prob{t > \t_s^+}\left(\frac32 + \e\right) n\log n \\
& = \left(\frac32 + \e\right)n\log\bfrac{3n}{3n-2\t_s^+ + 1} + o(n\log n)
\end{align}
Now, as $3n-2\t_s^+ = 3n\d_s\om$,
\begin{align}
\log\bfrac{3n}{3n-2\t_s^+ + 1} = (1+o(1))\log\bfrac{1}{\d_s\om} = \left(\frac23+o(1)\right)\log\bfrac{n}{n-s+1}.
\end{align}
The lower bound for $\E{C_V(s)}$ is found similarly.
\end{proof}
Substituting $s=n$ into \eqref{vcover} gives us \eqref{vertexcover}.
\section{Concentration: Proof of Lemma \ref{lem:concentration}}\label{sec:concentration}

In this section we prove bounds for $X_1(t), X_3(t), \F(t)$, collected in Lemma \ref{lem:concentration}. The following lemma contains a proof of part (i).
\begin{lemma}\label{lem:secondmoment}
Let $\d_3 = n^{-2/3}\log^4 n$, setting $t_3 = (1-\d_3)\frac{3n}{2}$. Then
\begin{align}
\Prob{\exists 1\leq t\leq t_3 : X_3(t) > \frac54 n\d^{3/2}} & = o(1).
\end{align}
Let $\om$ tend to infinity with $n$. Then for any $\om^{-1} \geq \d \geq \om n^{-2/3}$,
$$
\Prob{|X_3(t) - n\d^{3/2}| > \frac{n\d^{3/2}}{\om^{1/2}}} = o(1),
$$
and for any $\d \leq \om^{-1}n^{-2/3}$,
$$
\Prob{X_3(t) > 0} = o(1).
$$
\end{lemma}
\begin{proof}
For any set $U$ of vertices with $|U| = k$, using calculations similar to \eqref{EX3},
\begin{align}
\Prob{U\subseteq  X_3(t)} & = \left(1 - \frac{k}{n}\right)\prod_{j=0}^{t-1} \left(1 - \frac{3k}{3n-2j-1}\right) \\
& \leq \bfrac{3n-2t}{3n}^{3k/2} = \d^{3k/2}.
\label{eq:probvinX3}
\end{align}
Here the $1-k/n$ factor accounts for the probability that the starting point of the walk is in $\ol{U}$. For any new edge $(x, y)$ that is added, $U$ can only be entered if the endpoint $y$ which is chosen uniformly at random from all $3n-2j-1$ available configuration points is in $U$. We have, using the notation $(x)_k = x(x-1)\dots(x-k+1)$,
\begin{align}
\E{X_3(t)} & =n\d^{3/2}\brac{1+O\bfrac{1}{\d n}}, \quad\text{ directly from \eqref{EX3}}.\label{bla1}\\
\E{(X_3(t))_2} & \leq n(n-1)\d^3, \quad\text{ from \eqref{eq:probvinX3} with }k=2. \label{bla2}\\
\E{(X_3(t))_k} & \leq (n)_k  \d^{3k/2} \leq n^k \d^{3k/2}\quad\text{ from \eqref{eq:probvinX3} in general.}
\end{align}
Firstly, it follows from \eqref{bla1} and \eqref{bla2} that we have $\Var{X_3(t)}=O(n\d^{3/2})=O(\E{X_3(t)})$, and so Chebyshev's inequality shows that for any $\d \geq\om n^{-2/3}$,
$$
\Prob{|X_3(t) - n\d^{3/2}| > \frac{n\d^{3/2}}{\om^{1/2}}}=O \bfrac{\E{X_3(t)}\om}{n^2\d^3}= \bfrac{\om}{\om^2} = o(1).
$$
Secondly, the Markov inequality and \eqref{bla1} shows that for $\d \leq \om^{-1}n^{-2/3}$,
$$
\Prob{X_3(t) \geq 1} \leq(1+o(1)) n\d^{3/2} = o(1).
$$

As $|X_3(t)|$ takes nonnegative integer values we can for any real $z > 1$, express the binomial theorem as
$$
\E{z^{X_3(t)}} = \sum_{k\geq 0} \frac{\E{(X_3(t))_k} (z-1)^k}{k!} \leq \sum_{k\geq 0} \frac{n^k \d^{3k/2}(z-1)^k}{k!} = \exp\left\{(z-1)n\d^{3/2}\right\},
$$
So for any positive $\th = o(1)$, the moment generating function of $X_3(t)$ satisfies
$$
\E{e^{\th X_3(t)}} \leq \exp\left\{\th n\d^{3/2}(1+o(1))\right\}.
$$
Let $\d \geq \d_3 = n^{-2/3}\log^4 n$. Then with $\th = 2/\log n$,
$$
\Prob{X_3(t) > \frac54n\d^{3/2}} \leq \frac{\E{e^{\th X_3(t)}}}{\exp\left\{\frac54\th n\d^{3/2}\right\}} \leq \exp\left\{-\frac14\th n\d^{3/2}(1-o(1))\right\} = o(n^{-1}).
$$
Summing over $t = 1,2,\dots,t_3$, it follows that $X_3(t) \leq \frac54 n\d^{3/2}$ for all $\d \geq n^{-2/3}\log^4 n$ w.h.p.
\end{proof}

Lemma \ref{lem:secondmoment} implies bounds for $X_1(t)$ via $X_1(t) = 3n-2t-1-2X_2(t) - 3X_3(t)$. Firstly, for $\th > 0$ with $\th = o(1)$, as $X_2(t) \leq 1$,
\begin{equation}\label{eq:X1mgf}
\E{e^{-\th X_1(t)}} \leq e^{-\th(3n-2t+ 2)} \E{e^{3\th X_3(t)}} \leq \exp\left\{-3\th n\d + 3\th n\d^{3/2}(1+o(1))\right\}.
\end{equation}
Secondly, $X_3(t) \leq \frac54 n\d^{3/2}$ holding for all $3n/4 \leq t \leq t_3$ w.h.p. implies that w.h.p.,

\begin{equation}\label{eq:X1lowerbound}
X_1(t) \geq 3n\d\left(1 - \frac54 \d^{1/2}\right)-3, \quad \frac{3n}{4} \leq t \leq t_3.
\end{equation}
For $n^{-1}\log^{11} n = \d_4 \leq \d \leq \d_3 = n^{-2/3}\log^4n$ we have $X_3(t) \leq \frac54 n\d_3^{3/2} = \log^6 n$ w.h.p., so w.h.p.
\beq{X1a}{
X_1(t) = 3n-2t-2X_2(t) - X_3(t) = 3n\d\brac{1-O\bfrac{\log^6n}{\d n}}=3n\d(1-o(1)).
}
Equations \eqref{eq:X1lowerbound} and \eqref{X1a} together imply Lemma \ref{lem:concentration}(ii) and (iii), assuming that $\om=o(\log n)$.

%We also need the following bound in Section \ref{sec:phaseone}: if $\d\leq 1/2$ then w.h.p, $X_1(t) \geq 3\left(1-\frac54 2^{-1/2}\right)n\d \geq n\d$.

Finally, we prove a lower bound for the number of green edges $\F(t)$.
\begin{lemma}\label{lem:Phiconcentration}
$$
\Prob{\exists t_1\leq t\leq t_3 : \F(t) < (\d\d_0)^{1/2}n} = o(1).
$$
\end{lemma}
Note that as $\d_0 = 1/\log\log n$, this implies that $\F(t) \gg n\d$ whenever $\d_1 \geq \d\geq \d_3$.

\begin{proof}
Fix some $t_1\leq t\leq t_3$, and define $\d = \d(t)$ by $t = (1-\d)\frac{3n}{2}$ as usual. Recall that $Y(t) = X_1(t)\setminus \{v_0\}$ where $v_0$ is the tail of the walk $W(t)$, as defined in Section \ref{sec:uniform}. Define events
$$
\EE(t) = \left\{Y(t) \geq \frac{3n-2t}{2}\right\},
$$
and let $\ind{t}$ denote the indicator variable for $\EE(t)$. Note that by \eqref{eq:X1mgf} with $\th=1/\log n$,
$$
\Prob{\ol{\EE(t)}} = \Prob{e^{-\th Y(t)} > e^{-\th\frac{3n-2t}{2}}} \leq \exp\left\{-\frac32\th n\d(1+o(1))\right\} \leq \eta
$$
for $\d\geq \d_3$, where we define $\eta = \exp\left\{-n^{1/3}\log^3 n\right\}$.
\begin{claim}\label{cl:mgf}
If $0 < \th \leq n^{-2/3}\log^{2}n$, then there exists an $\e = o(\log^{-1} n)$ such that for $t_0\leq t\leq t_3$,
$$
\E{e^{-\th(\F(t+1)-\F(t))}\ind{t}\  \middle|\ [W(t)]} \leq \exp\left\{\frac{\th\F(t)}{3n-2t}(1 + \e)\right\}\ind{t}.
$$
\end{claim}

We now show how Claim \ref{cl:mgf} is used to provide a lower bound for $\F(t)$, before proving the claim. Firstly, as each vertex of $Y(t)$ is incident to exactly two green edges, we have $\F(t_0) \geq Y(t_0)$. So for $\th > 0$ with $\th = o(1)$, we have by \eqref{eq:X1mgf},
\begin{align}
\E{e^{-\th\F(t_0)}\ind{t_0}} \leq \E{e^{-\th (X_1(t_0)-1)}} \leq \exp\left\{-3\th n\d_0(1 - o(1))\right\}. \label{eq:ft0bound}
\end{align}
Set
$$
f_t(\th) = \E{e^{-\th\F(t)}\ind{t}}.
$$
As $\th, \F(t) > 0$, we have $e^{-\th\F(t)} \leq 1$ and
$$
f_t(\th) = \E{e^{-\th\F(t)}\ind{t-1}} + \E{e^{-\th\F(t)}(\ind{t} - \ind{t-1})} \leq \E{e^{-\th\F(t)}\ind{t-1}} + \Prob{\EE(t)\setminus \EE(t-1)}.
$$
Note that $\Prob{\EE(t)\setminus \EE(t-1)} \leq \Prob{\ol{\EE(t-1)}} \leq \eta$. Claim \ref{cl:mgf} implies that
\begin{align}
f_{t+1}(\th) & = \E{e^{-\th \F(t+1)}\ind{t}} + \eta \\
& = \E{\E{e^{-\th\F(t)} e^{-\th(\F(t+1)-\F(t))}\ind{t}\ \middle|\ [W(t)]}} + \eta \\
& \leq \E{e^{-\th \F(t)} \E{e^{-\th(\F(t+1)-\F(t))}\ind{t}\ \middle|\ [W(t)]}} + \eta \\
& \leq \E{\exp\left\{-\th\F(t) + \frac{\th\F(t)}{3n-2t}(1 + \e)\right\} \ind{t}} + \eta\\
& = f_t\left(\th - \frac{\th(1 + \e)}{3n-2t}\right) + \eta. \label{eq:frecursion}
\end{align}
From the bound \eqref{eq:ft0bound} it follows by induction that for $t > t_0$,
$$
f_t(\th) \leq \exp\left\{-3\th n\d_0(1-o(1))\prod_{s = t_0}^{t-1} \left(1 - \frac{1+\e}{3n-2s}\right)\right\} + (t-t_0)\eta.
$$
With $\d_0 = 1/\log\log n, \d \leq \d_1 = \log^{-1/2}n$ and $\e = o(\log^{-1}n)$, we have by calculations similar to those leading up to \eqref{EX3},
$$
\prod_{s=t_0}^{t-1} \left(1 - \frac{1+\e}{3n-2s}\right) \approx \exp\left\{-\sum_{s=t_0}^{t-1} \frac{1}{3n-2s} - o\left(\frac{1}{\log n} \sum_{s=t_0}^{t-1} \frac{1}{3n-2s}\right)\right\} \approx \bfrac{3n-2t}{3n-2t_0}^{1/2},
$$
as $\sum_{s=t_0}^{t-1}1/(3n-2s) = O(\log n)$. This implies that for $\d_3 \leq \d \leq \d_1$,
\begin{align}
3n\d_0(1-o(1)) \prod_{s=t_0}^{t-1}\left(1 - \frac{1+\e}{3n-2s}\right) & = 3n\d_0\bfrac{3n-2t}{3n-2t_0}^{1/2}(1+o(1))\\
& \geq 2n(\d\d_0)^{1/2},
\end{align}
so
\begin{equation}\label{eq:ftbound}
f_t(\th) \leq \exp\left\{-2 \th n(\d\d_0)^{1/2}\right\} + (t-t_0)\eta
\end{equation}
Now, if
$$L(t)=n(\d\d_0)^{1/2}$$
then
\begin{align}
\Prob{\F(t) < L(t)} & \leq \Prob{\F(t) < L(t), \EE(t)} + \Prob{\ol{\EE(t)}}
\end{align}
and for $\th > 0$, the bound $\ind{\{X > a\}} \leq X/a$ for $X, a > 0$ implies
\begin{align}
\Prob{\F(t) < L(t), \EE(t)} = \E{\ind{\left\{e^{-\th \F(t)} > e^{-\th L(t)}\right\}}\ind{t}} \leq \frac{\E{e^{-\th \F(t)}\ind{t}}}{e^{-\th L(t)}}.
\end{align}
Then for $\d\leq \d_1=\log^{-1/2}n$ and $\th = n^{-2/3}\log^2 n$,
\begin{align}
e^{\th L(t)}\eta & \leq \exp\left\{\left(n^{-2/3}\log^2 n\right)\left(n(\d_1\d_0)^{1/2}\right) - n^{1/3}\log^3 n\right\} \\
&\leq \exp\left\{n^{1/3}\left(\frac{\log^{2-1/4} n}{\sqrt{\log\log n}} - \log^3 n\right)\right\} = o(n^{-2}),
\end{align}
and as $f_t(\th) \leq e^{-2 \th L(t)} + \frac{3n}{2}\eta$ by \eqref{eq:ftbound},
$$
\Prob{\F(t) < L(t)} \leq e^{\th L(t)}f_t(\th) + \eta \leq e^{-\th n(\d\d_0)^{1/2}} + o(n^{-1}).
$$
We conclude that if $\th = n^{-2/3}\log^2 n$,
\begin{align}
\Prob{\exists t_1 \leq t\leq t_3 : \F(t) < L(t)}
& \leq o(1) + \sum_{t=t_1}^{t_3} \left(\exp\left\{-\th n(\d\d_0)^{1/2}\right\} + o(n^{-1})\right) \\
& \leq o(1) + O\left(n\exp\left\{-\frac{\log^{2 + 4/2}n}{\sqrt{\log\log n}}\right\}\right) \\
& = o(1).
\end{align}
It remains to prove Claim \ref{cl:mgf}.

\begin{proof}[Proof of Claim \ref{cl:mgf}]
We are interested in the distribution of $\F(t + 1) - \F(t)$, conditioning on the contracted walk $[W(t)]$. Write $[W(t)] = (\langle W(t)\rangle, Y(t))$, where $Y(t)\subseteq X_1(t)$ is the set of vertices visited exactly once by $W(t)$, as defined in Section \ref{sec:uniform}. Write $\langle W(t)\rangle = (x_1',x_2',\dots,x_{2s-1}')$.

Reveal $x_{2s}' = \m(x_{2s-1}')$. If $x_{2s}' \in \PP(X_3(t)\cup X_2(t))$ then one green edge is added, and no other green edges are visited before the next edge is to be added, so $\F(t + 1) = \F(t) + 1$.

On the other hand, if $x_{2s}'\in \PP(X_1(t))$, then the walk will proceed until it finds another configuration point of $P(t)$. We are considering a random walk on $\langle W(t)\rangle$ starting at $x_{2s}'$. Let $v_0$ denote the tail vertex of the walk. Then $\langle W(t)\rangle$ induces a graph $G(t)$ on $X_0\cup \{v_0\}$ with blue and green edges, in which each vertex of $X_0$ has degree $3$ while $v_0$ has degree $1,2$ or $3$. Each time the walk on $G(t)$ traverses a green edge, we reveal the length of the corresponding green bridge in $W(t)$. The walk ends either when (i) a green edge is traversed and revealed to contain a vertex of $Y(t)$, or (ii) the vertex $v_0$ is reached, assuming $v_0\in X_1\cup X_2$.

Suppose there are $\F(t)$ green edges in $W(t)$, and $\f = \F(t) - Y(t)$ green edges in $[W(t)]$. Let $e_1,\dots,e_\f$ denote the green edges of $[W(t)]$. By Lemma \ref{lem:uniform}, the number of edges in the corresponding green bridges in $W(t)$ is a vector $(K_1,\dots,K_\f)$, uniformly drawn from all vectors with $K_i\geq 1$ and $\sum K_i = \F$. For $\ell \in \{1,\dots,\f\}$ we have
$$
\Prob{K_i = 1 \text{ for } i = 1,2,\dots,\ell} = \prod_{i=1}^\ell \frac{\binom{\F-i-1}{\f-i-1}}{\binom{\F-i}{\f-i}} = \prod_{i=1}^\ell \frac{\f-i}{\F-i} = \prod_{i=1}^\ell\left(1 - \frac{Y(t)}{\F(t)-i}\right)\leq \left(1 - \frac{Y(t)}{\F(t)}\right)^\ell.
$$
This shows that the number of green edges traversed before finding $v_0$ or a vertex of $Y(t)$ is stochastically dominated by a geometric random variable with success probability $Y(t) / \F(t)$.

This shows that in distribution, conditioning on the walk $[W(t)]$,
$$
\F(t+1) - \F(t) \stackrel{d}{=} 1 - B\left(\frac{X_1(t)}{3n-2t-1}\right) R_t
$$
where $B(p)$ denotes a Bernoulli random variable with success probability $p$, $R_t$ is stochastically dominated by a geometric random variable with success probability $Y(t) / \F(t)$, and the two random variables in the right-hand side are independent. We have
\begin{equation}\label{eq:mgfsplit}
\E{e^{-\th(\F(t + 1) - \F(t))}\mid [W(t)]} = e^{-\th}\left(1 - \frac{X_1(t)}{3n-2t-1} + \frac{X_1(t)}{3n-2t-1}\E{e^{\th R_t} \mid [W(t)]}\right).
\end{equation}
As $x\mapsto e^{\th x}$ is increasing, we can couple $R_t$ to a geometric random variable $Z_t$ with success probability $Y(t) / \F(t)$ so that
$$
\E{e^{\th R_t}\mid [W(t)]} \leq \E{e^{\th Z_t}\mid [W(t)]}.
$$

As $[W(t)]$ is such that $\EE(t)$ holds, i.e. $Y(t) \geq (3n-2t)/2 = \OM(n^{1/3}\log^4 n)$,
$$
|(1-e^{-\th})\E{Z_t}| \leq (\th + O(\th^2))\frac{\F(t)}{Y(t)} \leq \frac{\log^2 n}{n^{2/3}}\frac{3n/2}{\OM(n^{1/3}\log^4 n)} = o\bfrac{1}{\log n}.
$$
So, as $Z_t$ is geometrically distributed,
\begin{align}
\E{e^{\th Z_t}} = \frac{\frac{Y(t)}{\F(t)}e^{\th}}{1 - \left(1 - \frac{Y(t)}{\F(t)}\right)e^{\th}} = \frac{1}{1 - \frac{\F(t)}{Y(t)}(1 - e^{-\th})} & = 1 - \frac{\F(t)}{Y(t)}(1-e^{-\th}) + O\left(\frac{\F(t)^2}{Y(t)^2}(1-e^{-\th})^2\right) \\
& = 1 + \frac{\th \F(t)}{Y(t)} + O\bfrac{\th^2 \F(t)^2}{Y(t)^2} \\
& \leq 1 + \frac{\th\F(t)}{Y(t)} (1 + \e)
\end{align}
for some $\e = o(\log^{-1} n)$. So if $[W(t)]$ is a class of walks with $Y(t) \geq (3n-2t)/2$ then
\begin{align}
\E{e^{-\th(\F(t+1) - \F(t))}\ind{t}\mid [W(t)]}
& \leq e^{-\th}\left(1 - \frac{X_1(t)}{3n-2t-1} + \frac{X_1(t)}{3n-2t-1}\left(1 + \frac{\th\F(t)}{Y(t)} + O\bfrac{\th^2\F(t)^2}{Y(t)^2}\right)\right) \\
& \leq 1 + \frac{\th\F(t)}{3n-2t} + O\bfrac{\th^2\F(t)^2}{(3n-2t)^2} \\
& \leq \exp\left\{\frac{\th\F(t)}{3n-2t}(1+\e)\right\}
\end{align}
where $\e = o(\log^{-1} n)$.
\end{proof}

This finishes the proof of Lemma \ref{lem:Phiconcentration}.
\end{proof}

\bibliographystyle{plain}

\end{document}